\pdfoutput=1 

\documentclass[11pt]{article}

\usepackage[utf8]{inputenc}		

\usepackage{amsmath, amsthm, amssymb, amstext, amsfonts}
\usepackage{graphicx, enumerate}
\usepackage{hyperref}
\usepackage{aliascnt}

\theoremstyle{plain}
\newtheorem{thm}{Theorem}[section]
\newaliascnt{lem}{thm}
\newtheorem{lem}[lem]{Lemma}
\aliascntresetthe{lem}
\newaliascnt{pro}{thm}

\aliascntresetthe{pro}
\newaliascnt{cor}{thm}
\newtheorem{cor}[cor]{Corollary}
\aliascntresetthe{cor}
\newaliascnt{problem}{thm}

\aliascntresetthe{problem}
\newaliascnt{que}{thm}

\aliascntresetthe{que}
\newaliascnt{con}{thm}

\aliascntresetthe{con}
\newaliascnt{clm}{thm}

\aliascntresetthe{clm}
\newtheorem*{thm1}{\autoref{thm:CycleTheorem}} 
\newtheorem*{thm2}{\autoref{thm:WeakCtwDualityThm}} 
\newtheorem*{thm3}{\autoref{thm:hyperbolic}} 
\newtheorem*{thm4}{\autoref{thm:treelength}} 
\newtheorem*{con1}{Conjecture}
\newtheorem*{twDualityThm}{Tree-width duality theorem} 

\theoremstyle{definition}
\newaliascnt{defi}{thm}
\newtheorem{defi}[defi]{Definition}
\aliascntresetthe{defi}
\newaliascnt{exm}{thm}

\aliascntresetthe{exm}

\def\COMMENT{}
\def\sub{\subseteq}

\newcommand \td {tree-decom\-po\-si\-tion}
\newcommand \tw {tree-width}
\newcommand \ttw {{\rm tw}}
\newcommand \ctw {con\-nect\-ed tree-width}
\newcommand \cttw {{\rm ctw}}
\newcommand \XC {{\mathcal C}[X]}
\let\doublebar=\| \def\size#1{{\doublebar#1\doublebar}}

\begin{document}
\title{Connected tree-width}
\author{Reinhard Diestel and Malte M\"uller}
\date{\today}
\maketitle

\begin{abstract}\noindent
The {\it \ctw\/} of a graph is the minimum width of a \td\ whose parts induce connected subgraphs. 
Long cycles are examples of graphs that have small \tw\ but large \ctw.
We show that a graph has small \ctw\ if and only if it has small \tw\ and contains no long geodesic cycle.

We further prove a connected analogue of the duality theorem for \tw: a finite graph has small \ctw\ if and only if it has no bramble whose connected covers are all large. Both these results are qualitative: the bounds are good but not tight.%
   \COMMENT{}

We show that graphs of \ctw\ $k$ are $k$-hyperbolic, which is tight, and that graphs of \tw~$k$ whose geodesic cycles all have length at most~$\ell$ are $\lfloor{3\over2}\ell(k-1)\rfloor$-hyperbolic. The existence of such a function $h(k,\ell)$ had been conjectured by Sullivan.
\end{abstract}

\section{Introduction}
Let us call a \td\ $(T, (V_t)_{t  \in T})$ of a graph $G$ \emph{connected} if its parts $V_t$ are connected in $G$. 
For example, the standard minimum width \td\ of a tree or a grid is connected. The \emph{\ctw} $\cttw(G)$ of $G$ is the minimum width that a connected \td\ of $G$ can have.

Much of the practical use of \td s, connected or not, derives from the fact that $G$ reflects the nested edge-separations of $T$ obtained by deleting a single edge: these correspond to nested vertex-separations of~$G$. If the \td\ and $G$ are connected, we also have a converse:\penalty-200\ then every (connected) subtree of~$T$ induces a connected subgraph of~$G$.%
   \COMMENT{}

We shall indicate below some contexts in which \ctw\ has been used for applications. The purpose of this paper, however, is to answer the natural first question one would ask: does \ctw\ differ from ordinary \tw, and if so how? The answer will be unexpectedly satisfying: we shall find an obstruction to when the two parameters are tied (in the sense that each is bounded by a function of the other), and be able to show that it is the only obstruction. Let us make this precise.

The ordinary tree-width $\ttw(G)$ of a graph $G$ is clearly at most its \ctw, so large \tw\ causes a graph to have large \ctw. But it is not the only possible cause.

A $k$-cycle~$C$, for example, has tree-width~2 but connected \tw~$\left\lceil k/2\right\rceil$. Indeed, its only connected subgraphs are its segments. In any \td\ none of whose parts contains another, the intersection of two adjacent parts separates them in the whole graph~\cite[Lemma\,12.3.1]{DiestelBook10noEE}. But the intersection of two such segments of $C$ only separates them in~$C$ if they cover~$C$~-- in which case one of them has at least $\left\lceil k/2\right\rceil$ edges.

Long cycles as subgraphs do not raise the \ctw: consider wheels. But long geodesic cycles do: we shall be able to show that every graph containing a $k$-cycle geodesically also has connected \tw\ at least~$\lceil k/2 \rceil$ (\autoref{lem:geoCycle2Bramble}).%
   \COMMENT{}
   (A~subgraph $H\sub G$ is \emph{geodesic} if $d_H(x, y) = d_G(x, y)$ for all its vertices $x$ and~$y$, where $d(x,y)$ is the length of a shortest $x$--$y$ path in $H$ or~$G$, respectively. Note that geodesic $u$--$v$ paths are simply shortest $u$--$v$ paths.)

Our main theorem says that, conversely, large \tw\ and long geodesic cycles are the only two obstructions to small \ctw:

\begin{thm}\label{thm:CycleTheorem}
The \ctw\ of a graph $G$ is bounded above by a function of its \tw\ and of the maximum length of its geodesic cycles.

Specifically, if $G$ is not a forest, $\ttw(G) < k\in\Bbb N$, and $\ell$ is the maximum length of a geodesic cycle in~$G$, then $\cttw(G) < f(k,\ell)$ for
\begin{equation*}
f(k,\ell) = k +  \binom{k}{2} \big(\ell (k-2) - 1\big).
\end{equation*}
\end{thm}

\autoref{thm:CycleTheorem} is qualitatively best possible in that the two obstructions are independent:  a large cycle (as a graph) contains a large geodesic cycle but has small \tw, while a large grid has large \tw\ but all its geodesic cycles are small. And both graphs have large \ctw.

\medbreak

Among the classical%
   \footnote{See~\cite{DiestelOumDualityI} for a more recent duality theorem for \tw\ that is also tight.}
   obstructions to small ordinary \tw\ there is one that gives a tight duality theorem: large-order brambles.
A \emph{bramble} is a set of pairwise touching connected sets of vertices, where two vertex sets \emph{touch} if they share a vertex or the graph has an edge between them. A~set of vertices \emph{covers} (or is a \emph{cover} of) a bramble if it has a vertex in each of these sets. The \emph{order} of a bramble is the smallest size of a cover.

\begin{twDualityThm}[Seymour \& Thomas \cite{ST1993GraphSearching, DiestelBook10noEE}]
A graph has \tw\ at least $k\ge 0$ if and only if it contains a bramble of order $> k$.
\end{twDualityThm}

To adapt this duality to \ctw, let the \emph{connected order} of a bramble be the least order  of a connected cover, a cover spanning a connected subgraph. Given any \td\ of a graph, it is well known and easy to show that every bramble is covered by a part of that \td. Hence graphs of \ctw\ $< k$ cannot have brambles of connected order $> k$. \autoref{thm:CycleTheorem} will enable us to prove a qualitative converse of this:

\begin{thm}\label{thm:WeakCtwDualityThm}
There is a function $g\colon \mathbb{N} \rightarrow \mathbb{N}$ such that every graph with no bramble of connected order $> k$ has \ctw\ $< g(k)$.
\end{thm}

We conjecture that, just as in the case of ordinary \tw, the duality is also quantitatively tight:

\begin{con1}
A graph has \ctw\ at least $k\ge 0$ if and only if it contains a bramble of connected order $> k$.
\end{con1}

We believe that the notion of connected tree-width is natural enough to merit further study: \td s, after all, are meant to exhibit the tree-like structure of a graph, and this aim is better served when its parts are connected (as in a minimum-width decomposition of a tree) than if not (as in a minimum-width decomposition of a long cycle). However, let us give a few pointers to how it relates to other contexts, both in graph theory and beyond.

One of these is hyperbolic graphs. These are important especially when they are infinite and appear as Cayley graphs of hyperbolic groups. But hyperbolicity has also been exploited for the design of algorithms.%
   \COMMENT{}
   We shall prove the following (see \autoref{sec:hyperbolicity} for a more detailed statement):

\begin{thm}\label{thm:hyperbolic}
Graphs of \ctw~$k$ are $k$-hyperbolic. This is best possible for every~$k > 1$.
\end{thm}

Another width parameter, somewhat related to \ctw, was introduced recently by Dourisboure and Gavoille~\cite{tree-length_DG2007} and has since received some attention. They call it the {\em tree-length} of a graph~$G$: the smallest value, minimized over all its \td s, of the maximum distance in $G$ of any two vertices in a common part of this decomposition. Clearly, graphs of \ctw~$<k$ also have tree-length~$<k$,%
   \COMMENT{}
   so by \autoref{thm:CycleTheorem} the tree-length of $G$ is bounded by the same function $f(k,\ell)$ as its \ctw. In fact, Reidl and Sullivan~\cite{ReidlSullivan2013, AdcockMahoneySullivan2014} observed that the following better bound follows directly from of our lemmas for the proof of \autoref{thm:CycleTheorem}:

\begin{thm}\label{thm:treelength}
If $G$ has tree-width~$<k$ and no geodesic cycle longer than~$\ell$, and $G$ is not a forest, then the tree-length of~$G$ is at most $\ell(k-2)$.
\end{thm}

Chepoi et al.~\cite{Chepoietal2008_hyperbolic} showed that graphs of tree-length at most~$k$ are $4k$-hyperbolic. We shall give a simple direct proof showing that they are $\lfloor{3\over2}k\rfloor$-hyperbolic, which is best possible.

\autoref{thm:treelength} thus implies that graphs of \tw~$<k$ that have no geodesic cycle longer than~$\ell$ are $\lfloor{3\over2}\ell(k-2)\rfloor$-hyperbolic. This confirms the conjecture of Sullivan~\cite{ReidlSullivan2013} that the hyperbolicity of graphs can be bounded in terms of their \tw\ and the maximum length of their geodesic cycles. Note that bounding just one of these parameters will not imply hyperbolicity: long cycles have small \tw\ and grids have no geodesic cycles of length~$>4$, but neither of these (classes of) graphs is hyperbolic.

Finally, connected tree-width has been used directly for the design of algorithms. Jegou and Terrioux~\cite{JegouTerrioux2014preprint, JegouTerrioux2014} introduced it, initially independently and unaware of an earlier preprint of this paper~\cite{Malte2012}, in the context of constraint satisfaction problems in AI. They show how practical CSP algorithms can be improved if the constraint network considered has bounded \ctw. They also study the complexity of computing the \ctw\ of a graph.

We use the terminology of \cite{DiestelBook10noEE}. In particular, the {\em length\/} of a path is its number of edges, and the {\em distance} $d_G(u,v)$ between vertices $u,v$ in a graph~$G$ is the minimum length of a $u$--$v$ path in~$G$. We also assume familiarity with the basic theory of \td s as described in~\cite[Ch.\,12.3]{DiestelBook10noEE}.
The proof of our main result, \autoref{thm:CycleTheorem}, will be given in Sections \ref{sec:SectionNavs}--\ref{sec:MainProof}. In \autoref{sec:SectionDuality} we use it to prove \autoref{thm:WeakCtwDualityThm}. Theorems \ref{thm:hyperbolic}--\ref{thm:treelength} are proved in Section~\ref{sec:hyperbolicity}.

\section{Making a \td\ connected}\label{sec:SectionNavs}%
   \COMMENT{}

Our approach to obtaining an upper bound for the \ctw\ of a graph~$G$ will be to start with an ordinary \td\ of $G$ of low width, and then to make its parts connected by adding paths joining its components in~$G$. Since adding vertices to a part may invalidate axiom (T3) of a \td, we may have to add the same vertices to other parts too. Our task will thus be to ensure that the number of vertices added to any given part in this way, either explicitly in order to decrease its number of components or implicitly when repairing~(T3), remains bounded.

\begin{figure}[htb]
 \centering
 \includegraphics{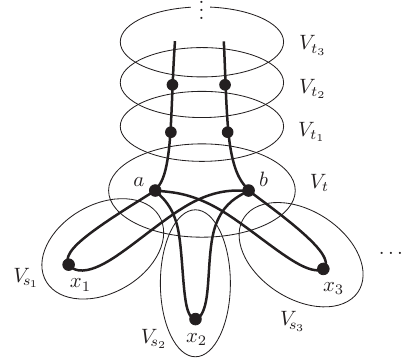}
 \caption{Making $V_{t_1}, V_{t_2},\dots$ connected can inflate $V_{t}$ unboundedly}
 \label{fig:NavReason}\vskip-3pt
\end{figure}

For example, consider the graph and \td\ indicated in \autoref{fig:NavReason}. If for $i=1,2,\dots$ we use the path through~$x_i$ to make $V_{t_i}$ connected, we will have to add all the $x_i$ to the central part~$V_{t}$ while repairing~(T3), because $t$ lies between $s_i$ and $t_i$ in~$T$.

Of course we made a bad choice here: we could use the path through~$x_1$ for every~$V_{t_i}$. This is the idea behind the following definition: if we already know a path joining two vertices $a$ and~$b$, we can re-use it whenever we have to connect $a$ and~$b$ later.

\begin{defi}[Navigational path systems, or {\em navs}]
Let $G=(V,E)$ be a graph, $K \subseteq [V]^{2}$ a set of 2-element subsets $xy :=\{x,y\}$ of its vertex set, and $\mathcal N = (P_{xy})_{xy\in K}$ a family of paths between these, where $P_{xy}$ links $x$ to~$y$.
   \begin{enumerate}[\rm (i)]\itemsep=0pt
   \item $\mathcal N$ is a \emph{$K$-nav} if for every path $P\in\mathcal N$ and for every two vertices $a,b\in P$ we have $ab\in K$ and $P_{ab} = aPb$.
   \item
   A \emph{nav} is a $K$-nav for $K = [V]^{2}$.
   \item
When $\mathcal D := (T, (V_t)_{t\in T})$ is a \td\ of $G$, then a $K$-nav satisfying $[V_t]^{2} \subseteq K$ for every $t\in T$ is a  \emph{$\mathcal D$-nav}.%
   \COMMENT{}
   \item
$\mathcal N$ is  \emph{geodesic} if every $P_{xy}\in\mathcal N$ is a shortest $x$--$y$ path in~$G$.
\item The \emph{length} $\ell(\mathcal N)$ of $\mathcal N$ is the maximum length of a path in~$\mathcal N$.
   \end{enumerate}
\end{defi}

\goodbreak

We shall use navs to make \td s connected, as follows:

\begin{lem} \label{lem:nav2conTreeDec}
Let $G$ be a graph, $\mathcal D = (T, (V_t)_{t\in T})$ a \td\ of $G$ of width~$<k$, and $\mathcal N = (P_{uv})_{uv\in K}$ a $\mathcal D$-nav of $G$. For all $t \in T$ let 
   $$W_t\> :=\ V_t\ \cup\!\!\!\bigcup_{xy\in[V_t]^{2}}\!\! V(P_{xy})$$
Then $(T, (W_t)_{t \in T})$ is a connected \td\ of~$G$ of width less than
   $$k + \binom{k}{2}  \big(\ell(\mathcal N) - 1\big).$$
\end{lem}

\begin{proof}
The family $(T, (W_t)_{t \in T})$ satisfies the axioms (T1) and~(T2) for \td s, because $\mathcal D$ does.
To prove (T3) let $t_1$, $t_2$ and $t_3$ be distinct nodes of $T$ with $t_2 \in t_1Tt_3$, and let $z\in W_{t_1} \cap W_{t_3}$ be given. We have to show that $z \in W_{t_2}$.

If $z\in V_{t_2}$ this is the case, so we assume that $z\notin V_{t_2}$. Then $V_{t_2}$, which separates $V_{t_1}$ from $V_{t_3}$ in~$G$, also separates $z$ from one of these, say from~$V_{t_3}$. Hence as $z\in W_{t_3}$ there are $x,y\in V_{t_3}$ such that $z\in P_{xy}$, and $z$ lies on a subpath $u P_{xy} v$ with $u,v\in V_{t_2}$. Since $\mathcal N$ is a $\mathcal D$-nav, this subpath coincides with $P_{uv}\sub W_{t_2}$, so $z\in W_{t_2}$ as desired.

The sets $W_t$ are connected in $G$ by construction. Since $|V_t|\le k$ by assumption, and each path in~$\mathcal N$ has at most $\ell(\mathcal N)-1$ inner vertices, they also satisfy $|W_t|\le k + \binom{k}{2} (\ell(\mathcal N) - 1)$.
\end{proof}

\autoref{lem:nav2conTreeDec} reduces our task of making a given \td~$\mathcal D$ connected without increasing its width too much to that of finding a $\mathcal D$-nav of bounded length. This will be achieved by a geodesic nav of~$G$.

While the existence of some nav in a connected graph is trivial~-- for example, we can route all its paths through some fixed spanning tree~-- the existence of a geodesic nav is more surprising. Spanning trees, for example, are never geodesic unless the entire graph is a tree.

\begin{lem} \label{lem:ExGeoNav}
Every connected graph has a geodesic nav.
\end{lem}

\begin{proof}
Let $G = (V,E)$ be a connected graph. Fix a linear ordering on~$V$, and consider the following ordering on the subsets of~$V$. If $v$ is the first vertex of $V$ that lies in the symmetric difference of two sets $U_1,U_2\sub V$, let $U_1 < U_2$ if $v\in U_2$. Then $U'\subsetneq U$ implies $U' < U$, and if we replace in $U$ a subset $U'\sub U$ with a smaller set $U'' < U'$, the resulting set $(U\setminus U')\cup U''$ is smaller than~$U$.%
   \COMMENT{}

For every two vertices $x,y\in V$ let $P_{xy}$ be a%
   \COMMENT{}
   shortest $x$--$y$ path in~$G$ whose vertex set is smallest in our ordering among all such paths. We claim that $\mathcal N := (P_{xy})_{xy\in [V]^2}$ is a nav in~$G$; if so, it will be geodesic by definition.

Suppose not. Then there are $x,y\in V$ and $a, b \in P_{xy}$ such that $P_{ab}$ differs from~$aP_{xy}b$. As both these paths are geodesic, it is easy to see that their vertex sets differ too;%
   \COMMENT{}
   then $V(P_{ab}) < V(aP_{xy}b)$ by the choice of~$P_{ab}$. By our earlier observation, replacing $aP_{xy}b$ with $P_{ab}$ in $P_{xy}$ yields a shortest $x$--$y$ path%
   \COMMENT{}
   with a vertex set smaller than~$V(P_{xy})$, contrary to the choice of~$P_{xy}$.\looseness=-1
\end{proof}

A geodesic nav $\mathcal N = (P_{xy})_{xy\in [V]^2}$ of $G$ induces, for any \td\ $\mathcal D = (T, (V_t)_{t\in T})$ of~$G$, a geodesic $\mathcal D$-nav $\mathcal N_{\mathcal D} := (P_{uv})_{uv\in K_{\mathcal D}}$ of~$G$ simply by restricting $[V]^2$ to its subset
 $$K_{\mathcal D} := \bigcup_{t \in T} \bigcup_{xy \in [V_t]^2} [V(P_{xy})]^{2};$$%
 note that $\mathcal N_{\mathcal D}$ is indeed a nav.
   \COMMENT{}
 Its length $\ell(\mathcal N_{\mathcal D})$ is bounded by the maximum distance of two vertices in a common part of~$\mathcal D$.

We have thus reduced our task further: it only remains to find a \td\ of width $\ttw(G)$ such that the distance between any two vertices in a common part is bounded by a function of $\ttw(G)$ and of the maximum length of a geodesic cycle in~$G$.

\section{Atomic \td s}\label{sec:SectionAtomicTD}

We shall work with \td s that cannot be refined in a certain sense. These were introduced by Thomas~\cite{thomas90} and have also been used in~\cite{bellenbaumDiestel}. We therefore only sketch the proofs of the few easy lemmas we need. Explicit proofs can be found in~\cite{Malte2012}.

\begin{defi}[Atomic \td]
The \emph{fatness} of a \td\ of a graph of order~$n$ is the $n$-tuple $(a_0, \ldots, a_n)$ in whichs $a_i$ denotes the number of parts of order~$n-i$. A~\td\ of lexicographically minimum fatness is \emph{atomic}.
\end{defi}

Clearly, atomic \td s of a graph $G$ have width $\ttw(G)$. The point of the notion is to minimize not just the largest part but the smaller ones too, in turn as their size decreases. 

Let $\mathcal D = (T, (V_t)_{t \in T})$ be an atomic \td\ of a graph~$G$.

\goodbreak
\begin{lem}\label{proper}
No part of $\mathcal D$ contains another part.
\end{lem}

\begin{proof}
If there are $t\ne t'$ such that $V_{t'}\subseteq V_t$, then by~(T3) such $t,t'$ can be chosen adjacent. Contracting the edge $tt'$ in $T$ and assigning $V_t$ as a part to the new contracted node yields a \td\ of smaller fatness than~$\mathcal D$, a contradiction.
\end{proof}

Given an edge $e=t_1 t_2$ of~$T$ and $i\in\{1,2\}$, let $T_i$ be the component of $T-e$ containing~$t_i$, and put $U_i := \bigcup_{t\in T_i} V_t$ and $G_i := G[U_i]$. Then $\{U_1,U_2\}$ is a separation of~$G$, with $X:= U_1\cap U_2 = V_{t_1}\cap V_{t_2}$~\cite[Lemma\,12.3.1]{DiestelBook10noEE}. 

\begin{lem}\label{connected}
$G_1 - X$ and $G_2 - X$ have components $C_1$ and~$C_2$, respectively, in which every vertex from~$X$ has a neighbour.
\end{lem}

\begin{proof}%
   \COMMENT{}
If the neighbourhood of every component of $G_1 - X$ (say) is a proper subset of~$X$, we can obtain a \td\ of smaller fatness than~$\mathcal D$ by replacing its portion that decomposes $G_1$ with the decompositions that $\mathcal D$ induces on the graphs $G_1[V(C)\cup N(C)]$, where $C$ ranges over the components of $G_1 - X$. This decomposition has smaller fatness, because we lose the part~$V_{t_1}\supseteq X$ and gain only parts that are proper subsets of parts we lose.%
   \COMMENT{}
   \looseness=-1

More formally, we replace $T_1$ in $T$ with copies $T_{1,C}$ of~$T_1$, one for every~$C$, making the nodes $t_{1,C}$ corresponding to $t_1$ adjacent to~$t_2$. For every $C$ and every $t\in T_1$, we associate $V_{t_C} := V_t\cap V(C)$ with the node $t_C$ of~$T_{1,C}$ corresponding to~$t$. It can be checked that the decomposition thus obtained has smaller fatness than~$\mathcal D$.%
   \COMMENT{}
\end{proof}

Lemma \ref{connected} implies at once:

\begin{cor}\label{findcycle}
Every two vertices $u,v\in X$ are linked by an $X$-\,path in~$G_1$ and another in~$G_2$, which thus form a cycle.\qed
\end{cor}

Our final lemma ensures that any two non-adjacent vertices in a common part of~$\mathcal D$ lie in some such~$X$. We shall then be able to use Corollary~\ref{findcycle} to find a cycle through them, which will be our aim when we prove Theorem~\ref{thm:CycleTheorem}.

\begin{lem}\label{thm:atoTDexIntersection}
For all $t\in T$ and distinct $u,v\in V_t$, either $uv\in E(G)$ or $t$~has a neighbour $s$ in~$T$ such that $u,v\in V_s$.
\end{lem}

\begin{proof}
If not, we split $V_t$ into $V_{t'}:= V_t\setminus\{u\}$ and $V_{t''}:= V_t\setminus\{v\}$, and replace $t$ in $T$ with new adjacent nodes~$t',t''$. This yields another \td\ (of smaller fatness) since, by assumption, every $V_s$ with $st\in E(T)$ fails to contain either $u$ or~$v$, so we can join $s$ to~$t'$ or~$t''$, respectively, in the expanded decomposition tree.
\end{proof}

\section{Use of the cycle space}\label{sec:SectionCycleSpace}

As shown at the end of Section~\ref{sec:SectionNavs}, our remaining task for a proof of Theorem \ref{thm:CycleTheorem} was to show that, in a suitable \td\ of width~${\ttw(G)<k}$, say,%
   \COMMENT{}
   the distance in $G$ between any two vertices in a common part is bounded by a function of $k$ and of the maximum length~$\ell$ of a geodesic cycle in~$G$.

The \td\ that we shall choose will be any atomic one, say ${\mathcal D} = (T, (V_t)_{t \in T})$. By Lemma~\ref{thm:atoTDexIntersection} we can then find, for any non-adjacent vertices $u,v$ in a common part of~$\mathcal D$, adjacent nodes $t_1,t_2\in T$ such that $\{u,v\}\subseteq X := V_{t_1}\cap V_{t_2}$. Corollary~\ref{findcycle} then provides $X$-paths $P_1\subseteq G_1$ and $P_2\subseteq G_2$ linking $u$ to~$v$ (the $G_i$ being defined as in Section~\ref{sec:SectionAtomicTD}), so that $C_{uv} := P_1\cup P_2$ will be a cycle.

If these cycles $C_{uv}$ were geodesic, the distances $d_G (u,v)$ would be bounded by~$\ell/2$ and achieve our goal. Unfortunately they are not. However, we shall use an algebraic argument to deduce from their existence that we can link the vertices~$u,v$ using only geodesic cycles meeting~$X$, and only a bounded number of these.%
   \COMMENT{}

Given $X\subseteq V(G)$ and a set $\mathcal C$ of cycles in~$G$, let us write
 $$\XC  := \bigcup\big\{ C\in {\mathcal C}\mid V(C)\cap X\ne\emptyset\big\}.$$
The set $\mathcal C$ will later be the set of geodesic cycles in~$G$. The {\em diameter} of~$X$ in~$G$ is the maximum distance in $G$ between any two vertices in~$X$.

\begin{lem}\label{lem:connectedClosure2distanceInG}
If $G[X]\cup\XC $ is connected, $1\le\ell\in\Bbb N$,%
   \COMMENT{}
   and every cycle in $\mathcal C$ has length at most~$\ell$, then $X$ has diameter at most $\ell\, (|X| - 1)$ in~$G$.
\end{lem}

\begin{proof}
If $\XC = \emptyset$ then $G[X]$ is connected, and the assertion holds. If not, then $\ell\ge 3$, which we shall now assume.%
   \COMMENT{}

Let $x_1,x_2\in X$ be given. By assumption, $G[X]\cup\XC $ contains an \hbox{$x_1$--$x_2$} walk~$W$ consisting of segments that are each either an edge of $G[X]$ or a path on a cycle in~$\XC $. By definition of~$\XC$, these latter segments, too, can be chosen so as to meet~$X$.%
  \COMMENT{}
   Cutting loops out of $W$ as necessary we may assume that no two of these segments meet $X$ in the same vertex~$x$, unless they are consecutive segments meeting at~$x$. 

Let us combine each pair of consecutive segments $u\dots x$ and $x\dots v$ with $x\in X$ and $u\notin X$ to a new segment $u\dots v$. This can be done simultaneously with all relevant segments, since by our conditions on $x$ and~$u$ a given segment can be combined with at most one of its two adjacent segments.%
   \COMMENT{}
   Choosing $u\dots x$ and $x\dots v$ as either an edge of~$G[X]$ or the shorter half of a cycle in~$\XC$, we can ensure that the combined segment has length at most~$\ell$.\looseness=-1

The resulting segments all meet $X$ and have distinct first vertices in~$X$. They each have length at most~$\ell$. The first segment, starting at~$x_1$, can in fact be chosen of length at most~$\ell/2$. So can the last segment, the one ending at~$x_2$, unless it is a combined segment, in which case $x_2$ is not its first vertex in~$X$. Altogether we now have either $|X|-1$ segments of length at most~$\ell$,%
   \COMMENT{}
   or $|X|-2$ inner segments of length at most~$\ell$ and two outer segments of length at most~$\ell/2$. So $W$ has length at most $\ell\, (|X|-1)$, as claimed.
\end{proof}

In our next lemmas we use + to denote addition in the binary cycle space of~$G$, that is, for the symmetric difference of edge sets. To avoid clutter, we shall not always distinguish notationally between a subgraph and its edge set, or between an edge $e$ and the set~$\{e\}$.

\begin{lem}\label{claimlemma}
Let $P=u\dots v$ be a path in $G$, where $uv=:e\notin E(G)$.%
   \COMMENT{}
   Let $D$ be an element of the cycle space of~$G$. Then $P + D$ contains a $u$--$\,v$~path.
\end{lem}

\begin{proof}
   \COMMENT{}
  Since $(P + e) + D$ lies in the cycle space of $G + e$, it is also an edge-disjoint union of cycles~\cite[Prop.\,1.9.1]{DiestelBook10noEE}. If~$C$ is the cycle containing~$e$, then $C-e$ is the desired \hbox{$u$--$v$}~path in $P + D$.
\end{proof}

\begin{lem}\label{lem:SeparatedCycle2ClosurePath}
Let $G = G_1\cup G_2$, and let $u,v\in X= V(G_1\cap G_2)$ be non-adjacent in~$G$. Assume that each~$G_i$ contains an $X$-\,path $P_i = u\dots v$, and that $\mathcal C$ is a set of cycles generating $C=P_1\cup P_2$. Then $\XC $ contains a $u$--$\,v$~path.
\end{lem}

\begin{proof}
As $\mathcal C$ generates~$C$, we can find subsets $\mathcal C_1, \mathcal C_X, \mathcal C_2\subseteq\mathcal C$ such that
 $$C=\bigoplus\mathcal C_1 + \bigoplus\mathcal C_X + \bigoplus\mathcal C_2\,,$$
 where every cycle in~$\mathcal C_i$ has its vertices in $G_i - X$ and every cycle in $\mathcal C_X$ meets~$X$. By Lemma~\ref{claimlemma} there is a $u$--$v$ path~$P$ in $P_1 + \bigoplus\mathcal C_1$. This latter set of edges is disjoint from $P_2 + \bigoplus\mathcal C_2$, so
 $$P\subseteq P_1 + \bigoplus\mathcal C_1 + P_2 + \bigoplus\mathcal C_2 = C + \bigoplus\mathcal C_1 + \bigoplus\mathcal C_2 = \bigoplus\mathcal C_X\subseteq \XC \,.$$\vskip-6pt
\end{proof}

\begin{lem}\label{lem:diameter}
Let $\mathcal D$ be any atomic \td\ of~$G$, and $1\le\ell\in\Bbb N$. Let $\mathcal C$ be a set of cycles in~$G$, all of length at most~$\ell$, that generate its cycle space. Then any two non-adjacent vertices in a common part $W\!$ of~$\mathcal D$ have distance at most~$\ell\cdot (|W|-2)$ in~$G$.
\end{lem}

\begin{proof}
Let ${\mathcal D} = (T, (V_t)_{t \in T})$. Let $W$ be any part of~$\mathcal D$, and let $u,v\in W$ be non-adjacent. We show that $d_G(u,v)\le\ell\cdot (|W|-2)$.

By Lemma~\ref{thm:atoTDexIntersection} we can find adjacent nodes $t_1,t_2\in T$ such that $\{u,v\}\subseteq X := V_{t_1}\cap V_{t_2}$. For every two non-adjacent vertices $u',v'$ in~$X$,%
   \COMMENT{}
   Corollary~\ref{findcycle} ensures that there are $X$-\,paths $P_1\subseteq G_1$ and $P_2\subseteq G_2$ linking $u'$ to~$v'$ (the $G_i$ being defined as in Section~\ref{sec:SectionAtomicTD}). By
\autoref{lem:SeparatedCycle2ClosurePath} there is a $u'$--$v'$ path in~$\XC$. Thus, $G[X]\cup\XC$ is connected.

By \autoref{lem:connectedClosure2distanceInG}, the vertices $u,v$ have distance at most~$\ell\cdot(|X|-1)$ in~$G$. By \autoref{proper} we have $|X| < |W|$, which completes the proof.
\end{proof}

\section{Proof of \autoref{thm:CycleTheorem}}\label{sec:MainProof}

We are now ready to prove our main result, which we restate.

\begin{thm1}
The \ctw\ of a graph $G$ is bounded above by a function of its \tw\ and of the maximum length of its geodesic cycles.

Specifically, if $G$ is not a forest, $\ttw(G) < k\in\Bbb N$, and $\ell$ is the maximum length of a geodesic cycle in~$G$, then $\cttw(G) < f(k,\ell)$ for
\begin{equation*}
f(k,\ell) = k +  \binom{k}{2} \big(\ell (k-2) - 1\big).
\end{equation*}
\end{thm1}

\begin{proof}
The \ctw\ of forests coincides with their ordinary \tw, so we assume that $G$ contains a cycle. Let $\ell$ denote the maximum length of a geodesic cycle in~$G$, and let $k\in\Bbb N$ be such that ${\ttw(G) < k}$. Assume without loss of generality that $G$ is connected.

By \autoref{lem:ExGeoNav}, $G$~has a geodesic nav~$\mathcal N$. Let ${\mathcal D} = (T, (V_t)_{t \in T})$ be any atomic \td\ of~$G$.  Then $\mathcal N$ induces a geodesic $\mathcal D$-nav $\mathcal N_{\mathcal D} := (P_{uv})_{uv\in K_{\mathcal D}}$ of~$G$, where
 $$K_{\mathcal D} := \bigcup_{t \in T} \bigcup_{xy \in [V_t]^2} [V(P_{xy})]^{2};$$%
 For each $t\in T$ let 
   $$W_t\> :=\ V_t\ \cup\!\!\!\bigcup_{xy\in[V_t]^{2}}\!\! V(P_{xy})\,.$$
By \autoref{lem:nav2conTreeDec}, $(T, (W_t)_{t \in T})$ is a connected \td\ of~$G$ of width less than
   $$k + \binom{k}{2}  \big(\ell(\mathcal N_{\mathcal D}) - 1\big).$$

It remains to show that $\ell(\mathcal N_{\mathcal D})\le \ell (k-2)$, i.e., that any two vertices $x,y$ in a common part of~$\mathcal D$ have distance at most~$\ell(k-2)$ in~$G$.%
   \COMMENT{}
   This is clear if $x,y$ are adjacent, since $G$ contains a cycle and hence $k,\ell\ge 3$. If $x,y$ are non-adjacent, it follows from Lemma~\ref{lem:diameter}, since the geodesic cycles of $G$ generate its cycle space~\cite[Ex.\,1.32]{DiestelBook10noEE} and $\mathcal D$ realizes the \tw\ of~$G$.
\end{proof}

\section{Tree-length and hyperbolicity}\label{sec:hyperbolicity}

The {\em tree-length\/} of a graph $G$ is the smallest value, minimized over all \td s of~$G$, of the maximum diameter in~$G$ of a part of this decomposition~\cite{tree-length_DG2007}. Reidl and Sullivan, see~\cite{ReidlSullivan2013, AdcockMahoneySullivan2014},%
   \COMMENT{}
   noticed that our lemmas from Sections \ref{sec:SectionAtomicTD} and~\ref{sec:SectionCycleSpace} can be applied directly to give a better bound on the tree-length of~$G$ than the bound of $f(\ell,k)$ implied by \autoref{thm:CycleTheorem}:

\begin{thm4}
If $G$ has tree-width~$<k$ and no geodesic cycle longer than~$\ell$, and $G$ is not a forest, then the tree-length of~$G$ is at most $\ell(k-2)$.
\end{thm4}

\begin{proof}
This follows from \autoref{lem:diameter}, as the geodesic cycles of $G$ generate its cycle space and every atomic \td\ realizes its \tw.
\end{proof}

Reidl and Sullivan point out that the graph obtained from the $n\times n$ grid by subdividing every every edge $m-1$ times has tree-width~$n$, no geodesic cycle longer than $\ell=4m$, and its tree-length at most~$nm$.%
   \footnote{In fact, it has \ctw~$nm$: the standard width~$n$ \td s of the $n\times n$ grid induce connected \td s of its subdivision of width~$nm$, and the well-known brambles of order~$n+1$ in the grid, see e.g.\ \cite[Ex.\,12.21]{DiestelBook10noEE}, extend to brambles of connected order~$nm+1$. See \autoref{sec:SectionDuality} for connected bramble order.}%
   \COMMENT{}
   The bound of $4m(n-1)$ offered by \autoref{thm:treelength} is thus off by no more than a factor of~4.

Given a real number $d\ge 0$, a graph is called {\em $d$-hyperbolic\/} if whenever vertices $x,y,z$ are linked by shortest paths $P_{xy} = x\dots y$ and $P_{yz} = y\dots z$ and $P_{xz} = x\dots z$, each of these paths lies within distance~$d$ of the other two. Trees, for example, are 0-hyperbolic, since $P_{xz}\sub P_{xy}\cup P_{yz}$.

Chepoi et al~\cite{Chepoietal2008_hyperbolic}%
   \COMMENT{}
   showed that graphs of tree-length~$d$ are $4d$-hyperbolic.%
   \footnote{Their result, Proposition~13, says `$d$-hyperbolic'. But they use a different notion of hyperbolicity that translates to ours at a loss of a factor of~4. See their Proposition~1.}
  It is not hard to make this bound sharp:

\begin{lem}\label{lem:hyperbolic}
Graphs of tree-length~$d$ are $\lfloor{3\over2}d\rfloor$-hyperbolic. This value is sharp for every $d\ge 0$.
\end{lem}

\begin{proof}
Let $G$  be a graph with a \td\ $(T, (V_t)_{t  \in T})$ into parts~$V_t$ each of diameter at most~$d$ in~$G$. Consider vertices $x,y,z\in G$ and shortest paths $P_{xy}, P_{yz}, P_{xz}$ between them as in the definition of hyperbolicity. We show that $P=P_{xz}$ lies within distance~$\lfloor{3\over2}d\rfloor$ of~$P' = P_{xy}\cup P_{yz}$.

Let $t_0\dots t_n$ be the path in~$T$ between a node $t_0$ such that $x\in V_{t_0}$ and a node $t_n$ with $z\in V_{t_n}$. Since $P$ is connected and meets $V_{t_0}$ as well as~$V_{t_n}$, it meets every set $S_i := V_{t_i}\cap V_{t_{i-1}}$ with $i=1,\dots,n$~\cite[Lemma~12.3.1]{DiestelBook10noEE}, in a vertex~$s_i\in S_i$ say. Similarly, $P'$ has a vertex $s'_i$ in every~$S_i$. Put $s_0:= x =: s'_0$ and $s_{n+1}:= z =: s'_{n+1}$, and let $S:= \{s_0,\dots,s_{n+1}\}$ and $S' := \{s'_0,\dots,s'_{n+1}\}$.

For each $i = 0,\dots,n$ we have $s_i,s_{i+1}\in V_{t_i}$ and hence $d(s_i,s_{i+1})\le d$. Since $P$ is a shortest path, its segments $s_i P s_{i+1}$ thus have length at most~$d$, so every vertex of $P$ lies within distance~$\lfloor d/2\rfloor$ of~$S$. Similarly, $d(s_i,s'_i) \le d$ for every~$i$, so $S$ lies within distance~$d$ of $S'\sub V(P')$. Hence $P$ lies within distance~$\lfloor{3\over2}d\rfloor$ of~$P'$, as claimed.

To show that the bound of $\lfloor{3\over2}d\rfloor$ is sharp, consider a path-decomposition of adhesion~2 into many copies of~$K^4$ and with disjoint adhesion sets. Subdividing every edge $d-1$ times yields a graph $G$ of tree-length at most~$d$:%
   \COMMENT{}
   the \td\ witnessing this is a path-decomposition whose parts are the 4-vertex-sets from the~$K^4$s together with pendent parts each consisting of a subdivided edge.

It is easy to find in $G$ three vertices $x,y,z$ with shortest paths $P_{xy}$, $P_{yz}$ and $P_{xz}$ between them such that each subdivdided~$K^4$ (other than the first and the last) meets $P_{xz}$ in exactly one subdivided edge~$P$ and $P_{xy}\cup P_{yz}$ in exactly one subdivided edge $Q$ disjoint from~$P$ (Fig.~\ref{fig:K4s}). For each subdivided~$K^4$ except the leftmost and the rightmost one, the vertex closest to the mid-point of $P\sub P_{xz}$ has distance exactly $\lfloor{3\over2}d\rfloor$ not only from~$Q$ but from all of~$P_{xy}\cup P_{yz}$.
\end{proof}\vskip-6pt

\begin{figure}[htb]
 \centering
 \includegraphics{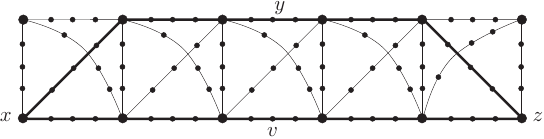}
 \caption{Geodesic paths between $x,y,z$ showing that $\lfloor{3\over2}d\rfloor$-hyperbolicity is best possible for tree-length~$d$}
 \label{fig:K4s}
\end{figure} 

\autoref{lem:hyperbolic} and \autoref{thm:treelength} imply:

\begin{cor}\label{cor:hyperbolicitybdd}
If $G$ has tree-width~$<k$, has no geodesic cycle longer than~$\ell$, and $G$ is not a forest, then $G$ is $\lfloor{3\over2}\ell(k-2)\rfloor$-hyperbolic.\qed
\end{cor}

If we strengthen the premise in \autoref{lem:hyperbolic} so as to assume that not just the tree-length but even the \ctw\ is bounded, we also get a stronger bound, which is again sharp:

\begin{thm3}
The following statements hold for every integer $k>1:$
\begin{enumerate}[\rm (i)]\itemsep=0pt
\item Every graph of \ctw~$k$ is $k$-hyperbolic.
\item There exists a graph of \ctw~$k$ that is not $(k-1)$-hyperbolic.
\end{enumerate}
The graphs of \ctw~$1$ are forests, which are $0$-hyperbolic.
\end{thm3}

\begin{proof}
Repeat verbatim the first two paragraphs of the proof of \autoref{lem:hyperbolic}, except that the \td\ $(T, (V_t)_{t  \in T})$ is now connected and of width~$k$, and we wish to show that $P$ lies within distance~$k$ of~$P'$. Let $v\in P$ be any vertex, say in its segment $P_i := s_i P s_{i+1}$.

Let $T_i$ be a spanning tree of~$G[V_{t_i}]$. Since $P$ is a geodesic path, $P_i$~is at most as long as the path $Q_i := s_i T_i s_{i+1}$. Let $R_i$ be the unique $s'_i$--$Q_i$ path in~$T_i$, and let $q_i\in Q_i$ be its last vertex.

Since $s_i P_i s_{i+1} Q_i s_i$ is a closed walk that has $\size{P_i} + \size{Q_i}\le 2\size{Q_i}$ edges and contains both $v$ and~$q_i$, we have $d(q_i,v)\le\size{Q_i}$. As $R_i$ is an $s'_i$--$q_i$ path with at most $\size{T_i} - \size{Q_i} \le k - \size{Q_i}$ edges,%
   \COMMENT{}
   we thus have $d(s'_i,v)\le k$.

Figure~\ref{fig:hyp} shows a graph $G$ of \ctw~$k=5$ that is not $(k-1)$-hyperbolic. Indeed, $G$~has a path-decomposition into parts of order at most~$k+1$ that form roughly vertical slices, each spanning a path in~$G$. Two (non-adjacent) such parts are shown in the figure.%
   \COMMENT{}
   To see that $G$ is not $(k-1)$-hyperbolic, consider the shortest paths $P_{xy}$, $P_{yz}$, $P_{xz}$ between $x$, $y$ and~$z$ whose union is the perimeter cycle, and the middle vertex $v$ of~$P_{xz}$. The example can be adapted to other values of $k\ge 3$ by extending the underlying ladder and altering the number of small vertices on its rungs, and to $k=2$ by keeping the rungs unsubdivided and contracting the edge below~$y$.
\end{proof}\vskip-9pt\vskip0pt

\begin{figure}[htb]
 \centering
 \includegraphics{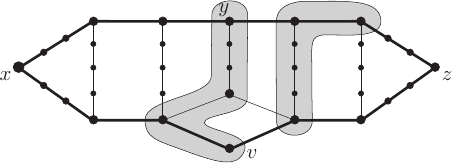}
 \caption{The vertex $v\in P_{xz}$ has distance $k$ from $P_{xy}\cup P_{yz}$}
 \label{fig:hyp}
\end{figure}

In the notation of the proof of \autoref{thm:hyperbolic}\,(i), replacing the segments $P_i\sub P$ with~$Q_i\sub T_i$ shows that, given a connected \td\ $(T, (V_t)_{t  \in T})$ of width~$<k$ of~$G$, the distance in $G$ between two vertices $v\in V_t$ and $v'\in V_{t'}$ is at most ${k\cdot (d_T(t,t')+1)}$.%
   \COMMENT{}
   It thus seems that graphs of bounded \ctw\ are quasi-isometric to their decomposition trees (and, in particular, hyperbolic). However, this is not the case, since the converse inequality may fail: there need not be a constant $c_k$%
   \COMMENT{}
   such that $d_T(t,t')\le c_k\cdot d_G(v,v')$, even if we choose the \td\ well.%
   \COMMENT{}

For example, let $G$ be obtained from a long path~$P$ by joining a new vertex~$v$ to every vertex of~$P$. All its \td s of low width, connected or not, have a long path in their decomposition tree~$T$: otherwise $T$ would have a node $t$ of large degree, and $V_t$ would be a small separator of $G$ leaving many components; but $G$ has no such separator. So while the distance of two vertices $v,v'\in P$ is always at most~2 in~$G$, the distance in $T$ of nodes $t,t'$ with $v\in V_t$ and $v'\in V_{t'}$ can be arbitrarily large.

\section{Duality}\label{sec:SectionDuality}

We have postponed until now a proof that not only do long cycles themselves have large \ctw\ (as we saw in the Introduction), but more generally that any graph with a long geodesic cycle does. Indeed, this is not entirely straightforward to prove. It is, however, an immediate consequence of \ctw\ duality, our topic in this last section.

Duality theorems for width parameters give us witnesses for large width. A~general framework for this, which includes all the classical duality theorems and several others, has recently been developed in~\cite{DiestelOumDualityI}. For \tw, the classical duality is between \td s and brambles:

\begin{twDualityThm}[Seymour \& Thomas \cite{ST1993GraphSearching,DiestelBook10noEE}]
A graph has \tw\ at least $k\ge 0$ if and only if it contains a bramble of order $> k$.
\end{twDualityThm}

\noindent
Rephrased as a min-max theorem: 1 + the minimum width of a \td\ equals the maximum order of a bramble.%
   \COMMENT{}

To adapt this duality to \ctw, let the \emph{connected order} of a bramble be the least order  of a \emph{connected cover}, a cover spanning a connected subgraph. (See the Introduction, or~\cite[Ch.\,12]{DiestelBook10noEE}, for definitions.)

\begin{con1}
Let $k \geq 0$ be an integer. A graph has \ctw\ at least~$k$ if and only if it contains a bramble of connected order $> k$.
\end{con1}

The proof of the backward implication, or of $\ge$ in the min-max version, is the same as for ordinary \tw. Given any bramble~$\mathcal B$ and any \td~$\mathcal D$, one shows that some part of~$\mathcal D$ covers~$\mathcal B$. Hence no bramble has greater order than the largest size of a part of~$\mathcal D$. And if $\mathcal D$ is connected, no bramble has greater connected order.

Let us apply this to show that long geodesic cycles raise the \ctw\ of a graph by giving rise to a bramble of large connected order.

\begin{lem} \label{lem:geoCycle2Bramble}
If a graph $G$ contains a geodesic cycle of length~$k\in\Bbb N$, then $G$ has a bramble of connected order at least~$\lceil k/2 \rceil + 1$, and $\cttw(G)\ge \lceil k/2 \rceil$.
\end{lem}

\begin{proof}
Let $C\sub G$ be a geodesic cycle of length~$k$, and let $\mathcal B$ be the set of all connected subsets of $V(C)$ of size $\lfloor k/2 \rfloor$. This is clearly a bramble. We shall prove that every connected cover $X$ of~$\mathcal B$ has at least $\lceil k/2 \rceil + 1$ vertices. This will also imply that $\cttw(G)\ge \lceil k/2 \rceil$, as noted above.

Choose $x,y\in X\cap V(C)$ at maximum distance $d_C(x,y) = d_G(x,y)$. If this distance is~$k/2\in\Bbb N$, then%
   \COMMENT{}
   $X$ contains an $x$--$y$ path with $\lceil k/2 \rceil + 1$ vertices.

If $d_C(x,y) < k/2$, then $X$ has another vertex $z$ in the longer $x$--$y$ segment of~$C$.%
   \COMMENT{}
  Let $C_{xy}, C_{yz}, C_{zx}$ be the shortest paths on~$C$ between their indices. By the choice of $x,y,z$, none of these contains another,%
   \COMMENT{}
   so $C_{xy}\cup C_{yz}\cup C_{zx} = C$. Hence $d(x,y) + d(y,z) + d(z,x)\ge k$, in $C$ and hence also in~$G$.

Consider a spanning tree~$T$ of~$G[X]$. The lenghts of the paths $xTy$, $yTz$ and~$zTx$ sum to at least $d(x,y) + d(y,z) + d(z,x)\ge k$. But no edge of $T$ lies on more than two of these paths, so $T$ has at least $\lceil k/2 \rceil$ edges. Thus, $|X| = |T|\ge \lceil k/2 \rceil + 1$ as desired.
\end{proof}

We already saw in the Introduction that $k$-cycles themselves have \ctw\ at least~$\lceil k/2 \rceil$. \autoref{lem:geoCycle2Bramble} says that graphs containing such cycles geodesically inherit this. This may lead one to suspect that, whenever $H\sub G$ geodesically, $\cttw(H)\le\cttw(G)$. The graph $G$ of Figure~\ref{fig:wheel}, however, shows that this is not so. Indeed, $G$ itself has a connected \td\ of width~3, with $v$ and its neighbours as a central part of order~4. Its geodesic subgraph $H=G-v$, however, has a bramble of connected order~5 (and hence $\cttw(H)\ge 4$): the set of paths of order~4 on its outer 9-cycle.%
   \COMMENT{}

\begin{figure}[htb]
 \centering
 \includegraphics{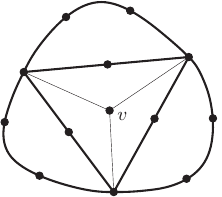}
 \caption{$G-v$ is geodesic in $G$ but has larger \ctw{}}
 \label{fig:wheel}
\end{figure} 

Still, it would be interesting to explore further the relationship between both the \ctw\ and the largest connected order of brambles of a graph and the corresponding parameters of its geodesic subgraphs. 

While deleting edges can increase the \ctw\ of a graph (for example, delete the inner rungs of a long ladder), it is easy to see that contracting edges cannot increase the \ctw.%
   \COMMENT{}

Let us finish by proving a qualitative connected version of the hard implication of the \tw\ duality theorem, to establish at least a weakening of our \ctw\ duality conjecture:

\begin{thm2}
There is a function $g\colon\mathbb{N} \rightarrow \mathbb{N}$ such that every graph of \ctw\ at least~$g(k)$ has a bramble of connected order~$>k$.
\end{thm2}

\begin{proof}
Let $g(k) = 2$ for $k\le 2$, and $g(k) = f(k,2k-2)$ for $k > 2$, where $f$ is the function from \autoref{thm:CycleTheorem}.

Let $G$ be a graph with no bramble of connected order $> k\in\Bbb N$. If $G$ is a forest then $\cttw(G) < 2\le g(k)$,%
   \COMMENT{}
   as desired. Assume now that $G$ is not a forest. Then $k > 2$, because cycles contain brambles of order~3.%
   \COMMENT{}

Since every bramble in $G$ has (connected) order~$\le k$,%
   \COMMENT{}
   we have $\ttw(G) < k$ by the classical \tw\ duality theorem, and every geodesic cycle in $G$ has length at most~$2k - 2$ by \autoref{lem:geoCycle2Bramble}.%
   \COMMENT{}
  By \autoref{thm:CycleTheorem} this implies $\cttw(G) < f(k,2k-2) = g(k)$, as desired.
\end{proof}

\section{Acknowledgement}

 We thank Daniel Wei\ss auer for pointing out an error in an earlier version of Lemma~\ref{connected}.

\newpage

\bibliographystyle{plain}
\bibliography{collective}

\end{document}